\documentclass[10pt,twoside]{article}

\usepackage{amsmath}
\usepackage{amssymb}
\usepackage{amsthm}
\usepackage[latin1]{inputenc}
\usepackage{eurosym}
\usepackage[dvips]{graphics}
\usepackage{graphicx}
\usepackage{epsfig}
\usepackage{hyperref}
\usepackage{ifthen}
\usepackage{setspace}
\usepackage[active]{srcltx}
\usepackage{pdfsync}
\usepackage{amsfonts}
\usepackage{color}

\newcommand{\Rr}{{\mathbb{R}}}
\newcommand{\Nn}{{\mathbb{N}}}

\newcommand{\D}[1]{\mbox{\rm #1}}
\newcommand{\dd}{\D{d}}

\newtheorem{theorem}{Theorem}[section]

\newtheorem{rem}[theorem]{\sc Remark}
    \newenvironment{remark}{\begin{rem} \begin{rm}}{\end{rm} \qed\end{rem}}
\newtheorem{lemma}[theorem]{Lemma}

 \newtheorem{proposition}[theorem]{Proposition}

\newtheorem{definition}[theorem]{Definition}

\theoremstyle{definition}

\numberwithin{equation}{section}

\makeatletter
\evensidemargin \oddsidemargin
\makeatother

\begin{document}
\thispagestyle{empty}
\setcounter{page}{1}


\begin{center}
{\large\bf \uppercase{Well-posedness for the Cauchy problem for a
fractional porous medium equation with variable density in one
space dimension}}

\vskip.20in

Fabio Punzo \\[2mm]
{\footnotesize
Dipartimento di Matematica ``G. Castelnuovo", \\
Universit\`{a} di Roma La Sapienza, \\ 
Piazzale Aldo Moro 5, 00185 Roma }\\ [3mm]
Gabriele Terrone\footnote{
Gabriele Terrone was supported by the UTAustin-Portugal partnership through the FCT post-doctoral fellowship
SFRH/BPD/40338/2007, CAMGSD-LARSys through FCT Program POCTI -
FEDER and by grants PTDC/MAT/114397/2009,
UTAustin/MAT/0057/2008, and UTA-CMU/MAT/0007/2009.
} \\[2mm]
{\footnotesize
Center for Mathematical Analysis, Geometry, and Dynamical Systems, \\ 
Departamento de Matem\'{a}tica, \\
Instituto Superior T\'{e}cnico, 1049-001 Lisboa, Portugal}
\end{center}

{\footnotesize
\noindent
{\bf Abstract.}
We study existence and uniqueness of bounded solutions to a fractional nonlinear porous medium equation with a variable density, in one space dimension.
\\[3pt]
{\bf Keywords.} Fractional Laplacian, Porous medium equation, Cauchy problem, Harmonic Extension, Variable density.
\\[3pt]
{\small\bf AMS subject classification:} 35A01, 35A02, 35E15, 35K55, 35R11.
}

\vskip.2in


\section{Introduction}
In this paper we study existence and uniqueness of bounded solutions to {\it nonlocal} nonlinear initial-value problems of the following type:
\begin{equation}
\label{06111}
    \begin{cases}
    \displaystyle\rho\, \partial_t u + \left(-\frac{\partial^2}{\partial x^2}\right)^{\frac{1}{2}}\left[ u^m\right] = 0  & x\in \Rr, \, t>0\\
     u=u_0 & x\in \Rr, \, t=0;
    \end{cases}
\end{equation}
here $\rho=\rho(x)$, usually referred to as a {\it variable density}, is a positive function only depending on the spatial variable $x,
 \left(-\frac{\partial^2}{\partial x^2}\right)^{\frac{1}{2}}$ denotes
the one-dimensional fractional Laplacian of order $1/2$, $m\geq 1$ and $ u_0$ is a nonnegative and bounded function.

The interest in the study of nonlocal problems has grown significatively in the last years and the analysis of nonlinear equations involving fractional powers of the Laplacian has become an area of intense research. Such problems in fact arise in many physical situations (see, $e.g.$ \cite{AC}, \cite{Ja} , \cite{JKO}) in particular in the analysis of long-range or anomalous diffusions. From a probabilistic point of view, the fractional Laplacian is the infinitesimal generator of a L\'evy process (see \cite{Bert}). Then, losely speaking, uniqueness for problem \eqref{06111} with $m=1$ corresponds to the fact that the L\'evy process associated to the operator $\frac 1{\rho} \left(-\frac{\partial^2}{\partial x^2}\right)^{\frac{1}{2}}$, starting from any point in $\Rr$, does not attain {\it infinity}.

\medskip

If the nonlocal operator $\left(-\frac{\partial^2}{\partial x^2}\right)^{\frac{1}{2}}$ in \eqref{06111} is replaced by the classical second order partial derivatives $-\frac{\partial^2}{\partial x^2}$ we get
\begin{equation}\label{ea1}
  \begin{cases}
    \displaystyle\rho\, \partial_t u - \frac{\partial^2 u^m}{\partial x^2}  = 0  & x\in \Rr, \, t>0\\
     u=u_0 & x\in \Rr, \, t=0,
    \end{cases}
\end{equation}
which is the initial value problem for the porous medium equation with variable density in one space dimension.
Problem \eqref{ea1} and its counterpart in $\Rr^N$, that is,
\begin{equation}
\label{ea2}
    \begin{cases}
    \displaystyle\rho\, \partial_t u + \Delta u^m  = 0  & x\in \Rr^N, \, t>0\\
     u=u_0 & x\in \Rr^N, \, t=0,
    \end{cases}
\end{equation}
occurs in various situations of physical interest (see, $e.g.$, \cite{KR}) and have been largely investigated 
in the literature during the last two decades (see \cite{Eid90}, \cite{EK},  \cite{GMP}, \cite{GHP}, \cite{KRV}, \cite{KKT}, \cite{P1}, \cite{RV1}-\cite{RV2}). More recently, similar problems on Riemannian manifolds have been  addressed in \cite{P2}-\cite{P4}.

It is well-known that the asymptotic behavior of the varying density $\rho$ may influence uniqueness of solutions for problem \eqref{ea2}. Briefly, for $N=1$ or $N=2$ problem \eqref{ea2} is well-posed in the class of bounded solutions not satisfying any additional conditions at infinity (see \cite{GHP}), provided $\rho\in L^\infty(\Rr^N)$. The situation is different if $N\geq 3$. Indeed, in this case uniqueness prevails when the density $\rho$ goes to zero slowly as $|x|\to\infty$. If instead $\rho$ goes to zero fast as $|x|$ diverges, then nonuniqueness of bounded solutions occurs (see \cite{EK}, \cite{GMP}, \cite{KKT},  \cite{P1}-\cite{P4}, \cite{RV1}-\cite{RV2}).

Concerning the regularity assumed for initial data and
solutions we point out that in \cite{KRV}, \cite{RV1}-\cite{RV2}
 {\it weak energy solutions} are dealt with, and $u_0$ is supposed to belong to 
 $L^1_{\rho}(\Rr^N)$, the space of measurable nonnegative functions $f$ satisfying
$\int_{\Rr^N} f\,\rho \dd x <\infty$. In \cite{EK}, \cite{GMP} \cite{KKT} and
\cite{P1}, instead, bounded initial data and so-called {\it very
weak solutions} are treated. One of the main issues in dealing
with very weak solutions lies in the fact that, in general,
these solutions do not have finite energy in the whole $\Rr^N$.

\medskip
Very recently,  in \cite{V2},  the {\it nonlocal} nonlinear
initial-value problem
\begin{equation}
\label{ea3}
    \begin{cases}
    \partial_t u + (- \Delta)^{\frac{\sigma}{2}}\left[ u^m\right] = 0  & x\in \Rr^N, \, t>0\\
     u=u_0 & x\in \Rr^N, \, t=0
    \end{cases}
\end{equation}
$(0<\sigma <2)$ has been studied. The particular case $\sigma=1$
has been considered in \cite{V1}; more precisely, existence,
uniqueness and properties of weak solutions to \eqref{ea3} have
been established assuming $u_0\in L^1(\Rr^N)$. Furthermore, problem 
\[
    \begin{cases}
    \partial_t u + \left(-\frac{\partial^2}{\partial x^2}\right)^{\frac{1}{2}}\left[ \log(1+u)\right] = 0  & x\in \Rr, \, t>0\\
     u=u_0 & x\in \Rr, \, t=0
    \end{cases}
\]
has been addressed in \cite{V4}. 
\medskip

The motivation in studying problem \eqref{06111} is twofold. In fact, on one hand, when $\rho\equiv1$, \eqref{06111} becomes problem \eqref{ea3} with $\sigma=N=1$, thus problem \eqref{06111} can be regarded as a generalization of problem \eqref{ea3} to the case in which a variable density $\rho$ is taken into account. On the other hand, problem 
\eqref{06111}  can also be considered as nonlocal version of problem \eqref{ea1}.  
\medskip

The aim of this paper is to prove well-posedness of problem \eqref{06111} in the class of bounded solutions not satisfying any additional conditions at infinity. We shall consider bounded initial data $u_0$ and, consequently, very weak bounded solutions (see Definition \ref{06114}). Let us mention that our results differ from those in \cite{V1} where $\rho$ is assumed to be identically equal to $1$. Moreover, also in the case $\rho\equiv 1$ considered in \cite{V1} uniqueness of very weak solutions is not proved, while uniqueness of weak energy solutions is shown (see \cite[Lemma 4.1]{V1}). Very weak solutions to problem \eqref{ea3} in the space $C\big([0,\infty); L^1_{\varphi}(\Rr^N) \big)$ ($\varphi$ being a suitable weight decaying at infinity) have been considered in \cite{BonfV} where uniqueness of such type of solutions is also shown provided $0<m<1$.

To the best of our knowledge, the uniqueness result presented in this paper is new also in the linear case $m=1$, that is for problem
\begin{equation}\label{ea5}
 \begin{cases}
    \displaystyle\rho\, \partial_t u + \left(-\frac{\partial^2}{\partial x^2}\right)^{\frac{1}{2}}\left[ u\right] = 0  & x\in \Rr, \, t>0\\
     u=u_0 & x\in \Rr, \, t=0.
    \end{cases}
\end{equation}
Some results for nonlocal linear parabolic equation with a variable density have been established in \cite{Chas}, but not for problem \eqref{ea5}.

\medskip

One can expect that if $N\geq 2$ the nonlocal problem
\[
 \begin{cases}
   \rho \,\partial_t u + (- \Delta)^{\frac{1}{2}}\left[ u^m\right] = 0  & x\in \Rr^N, \, t>0\\
     u=u_0 & x\in \Rr^N, \, t=0
    \end{cases}
\]
may exhibit uniqueness or nonuniqueness of solutions depending on the behavior at infinity of the density $\rho$. This would reflect in the nonlocal case the the situation described above for problem \eqref{ea2}.
However, some technical difficulties prevent us from adapting the methods used in this paper  to investigate uniqueness of very weak solutions in any space dimension; see Remark \ref{Rem1a}. We postpone to a forthcoming paper the proof of existence and uniqueness of solutions to problem
\begin{equation}
\label{ea4}
    \begin{cases}
   \rho\, \partial_t u + (- \Delta)^{\frac{\sigma}{2}}\left[ u^m\right] = 0  & x\in \Rr^N, \, t>0\\
     u=u_0 & x\in \Rr^N, \, t=0
    \end{cases}
\end{equation}
for any $0<\sigma <2$ and any $N\ge 1$. In this case, according with \cite{V1}-\cite{V2}, we shall consider weak energy solutions and completely different arguments.

\medskip
The paper is organized as follows.
In Section \ref{sec:aa} we introduce the setting of the problem and give the precise definition of solutions we are dealing with. We also recall some basic facts about fractional Laplacian, mainly its realization trough the harmonic extension, which permits to look at problem \eqref{06111} in a ``local way'' as a quasi-stationary problem with dynamical boundary conditions. In Section \ref{sec:bb} we prove existence of solutions to problem \eqref{06111}; see Theorem \ref{06116}. Even if the general strategy of the proofs in this Section goes along the same lines as in \cite{V1}-\cite{V2}, there are some differences here that will be expedient as we will point out in Remark \ref{Rem2a}. We should remark that all results in Section \ref{sec:bb} are valid in general for any $N\geq 1$. In subsequent Section \ref{sec:cc} we establish uniqueness of solutions; see Theorem \ref{tuni}.
The proof exploits the above mentioned realization of the fractional Laplacian combined with properties of a suitable family of test functions we construct in Lemma \ref{l1u}. Here the restriction $N=1$ will be crucial as explained in Remarks \ref{Rem1a} and \ref{Rem3a}.

\medskip
To conclude, let us mention that our techniques also applies to more general
fractional nonlinear diffusion equations with variable density, such as
\[
\rho\, \partial_t u + \left(-\frac{\partial^2}{\partial
x^2}\right)^{\frac{1}{2}}\left[ G(u)\right] = 0\quad x\in \Rr,
t>0,
\]
provided
$G$ satisfies suitable conditions as in \cite{EK}, \cite{P1}. 
Moreover, changing sign
bounded solutions could also be considered. However, for sake of simplicity, we limit ourselves
to the case of nonnegative solutions and $G(u)=u^m$ $(m\geq 1)$.


\section{Problem setting and assumptions}
\label{sec:aa} We recall that the nonlocal operator
$\left(-\frac{\partial^2}{\partial x^2} \right)^{\frac 1 2}$ is
defined for any function $\varphi$ belonging to the Schwartz class
by
\[
\left(-\frac{\partial^2}{\partial x^2} \right)^{\frac 1 2} \varphi=\frac 1{\pi}\; \mbox{P.V.} \int_{\Rr}\frac{\varphi(x)-\varphi(y)}{|x-y|^2}\dd y\,.
\]
As is well-known, $\left(\frac{\partial^2}{\partial x^2}
\right)^{\frac 1 2}$ can be also defined in many other ways; we refer
the reader to \cite{CS1}, \cite{CafS}, \cite{DPVal}, for a
comprehensive account on the subject. In particular, if $\varphi$ is a smooth and
bounded function defined in $\Rr$, we can consider its harmonic
extension $v$ to the upper half-space
\[
\Omega:= \Rr^{2}_+ = \{(x,y): \, x\in \Rr, \, y>0 \}.
\]
Hence (see \cite{CS1}, \cite{CafS},  \cite{V1}, \cite{V4}),
\[-\frac{\partial v(x,0)}{\partial y}=\left(-\frac{\partial ^2}{\partial x^2} \right)^{\frac 1 2} \varphi(x)\quad \textrm{for all}\;\; x\in \Rr.\]


\smallskip

In the sequel, we always make the following assumption:
\begin{equation}
\label{A0} \tag{{\bf A}$_0$}
\begin{cases}
\text{(i)} &\rho\in C(\Rr), \,\rho>0 \text{ in } \Rr;\\
\text{(ii)}& u_0\,\, \text{is nonnegative, continuous and bounded
in }\Rr.
\end{cases}
\end{equation}

Following \cite{V1} and \cite{V2}, using the harmonic extension,
we shall rewrite the problem \eqref{06111} in terms of local
differential operators. More precisely, solving problem
\eqref{06111} is equivalent to solve the quasi-stationary problem
with a dynamical boundary conditions:
\begin{equation}
\label{06112}
    \begin{cases}
    \Delta w = 0 & (x,y)\in \Omega, \, t >0\\
    \displaystyle \frac{\partial w}{\partial y}= \rho\, \frac{\partial\left[ w^\frac{1}{m}\right]}{\partial t} & x\in \Gamma,\, t>0\\
     w=u_0^m & x\in \Gamma, \, t=0;
    \end{cases}
\end{equation}
here $\Gamma:=\overline \Omega \cap \{y=0\}\equiv \Rr$.

Now, take any open bounded subset $D\subset \overline \Omega$ with
$\partial D \cap \Omega$ smooth and $\partial D\cap \Gamma \neq
\emptyset$. Then the exterior normal to $\partial D$ exists almost
everywhere; we denote such vector by $\vec\nu$. Let $T>0, \psi\in
C^{2,1}_{x,t}(\overline D\times [0,T])$, $\psi\ge 0$, $\psi =0$ in
$(\partial D\backslash \Gamma)\times [0,T]$, $\psi \not\equiv 0$
in $(\partial D\cap \Gamma)\times [0,T]$.
Formally, multiplying the differential equation in \eqref{06112}
by $\psi$, integrating by parts and considering the initial and
the dynamical boundary condition  we get:

\[
0   = \int_0^T \int_{D} \psi \,\Delta w \,\dd x\, \dd y\, \dd t \]
  \[ = - \int_0^T \int_{D} \langle\nabla\psi, \nabla w \rangle\,\dd x\, \dd y\, \dd t
           + \int_0^T \int_{\partial D} \psi\,   \frac{\partial w}{\partial\vec{\nu}}\,\dd x\, \dd y\, \dd t\]
 \[= - \int_0^T \int_{D}\langle\nabla\psi, \nabla w \rangle \, \dd x\, \dd y\, \dd t
- \int_0^T \int_{\partial D\cap \Gamma} \psi \,   \frac{\partial w}{\partial y}\,
{\dd x\, 
\dd t
}
\]
\[   = \int_0^T \int_{D} w \,\Delta \psi  \,\dd x\, \dd y\, \dd t
            - \int_0^T \int_{\partial D\cap \Omega} w \frac{\partial \psi}{\partial \vec\nu}\, \dd S\,  \dd t
            + \int_0^T \int_{\partial D\cap \Gamma}  w \frac{\partial \psi}{\partial y}\, \dd x\, \dd t \]
\[      - \int_0^T \int_{\partial D\cap \Gamma}  \rho\,\psi\,  \partial_t u \,\dd x
\, \dd t \]
\[ =  \int_0^T \int_{D} w\, \Delta \psi  \,\dd x\, \dd y\, \dd t
        -   \int_0^T \int_{\partial D\cap \Omega} w \frac{\partial \psi}{\partial \vec\nu}\,\dd S \, \dd t
        + \int_0^T \int_{\partial D\cap \Gamma}  w \frac{\partial \psi}{\partial y}\, \dd x\,  \dd t\]
        \[ + \int_0^T \int_{\partial D\cap \Gamma} \rho\, u\, \partial_t \psi \,\dd x\, \dd t\]\[     - \int_{\partial D\cap \Gamma} \rho(x) \Bigl[ u(x,T)\psi(x,0,T) - u_0(x)\psi(x,0,0)\Bigr] \,\dd x.\]
    
Given a function $f\in W^{1,2}_{loc}(\overline \Omega)$ we denote by $f |_{\Gamma}$ its trace on $\Gamma$, which is in $L^2_{loc}(\Gamma)$.
In view of previous equalities we give next

\begin{definition}
\label{06114}
A \emph{ solution} to problem \eqref{06112} is a
pair of functions $(u,w)$ such that
\begin{itemize}
\item{}   $u \in L^{\infty} (\Gamma\times
(0,\infty))$, $u\ge 0$;
\item{} $w\in L^{\infty} (\Omega\times
(0,\infty))$, $w\ge 0$,  $\nabla w\in L^2_{loc}\big(\overline\Omega\times(0,\infty)
\big)$;
\item{} $w|_{\Gamma\times (0,\infty)}=u^m$;
\item{} for any $T>0, D$ and $\psi$ as in \eqref{291101} there holds
 \begin{multline}
    \label{06115}
    \int_0^T \int_{D} w \,\Delta \psi  \,\dd x\, \dd y\, \dd
    t -\int_0^T \int_{\partial D\cap \Omega} w \frac{\partial \psi}{\partial \vec\nu}\, \dd S \, \dd t  \\
    + \int_0^T \int_{\partial D\cap \Gamma}  u^m \frac{\partial \psi}{\partial y}\, \dd x\, \dd t
      + \int_0^T \int_{\partial D\cap \Gamma} \rho\, u\, \partial_t \varphi \,\dd x\, \dd
      t\\
     -\int_{ \Gamma} \rho(x) \Bigl[ u(x,T)\varphi(x,0,T) - u_0(x)\varphi(x,0,0)\Bigr] \,\dd x=0.
     \end{multline}
\end{itemize}
\end{definition}

Here and hereafter we say that $f\in L^2_{loc}\big(\overline\Omega\times(0,\infty)$ when $f\in L^2\big(\Omega'\times (0,\infty) \big)$ for every open subset $\Omega'\subset\subset
\overline\Omega$.
Note that, since $\nabla w\in L^2_{loc}\big(\overline\Omega\times(0,\infty)
\big)$, the equality $w|_{\Gamma\times (0,\infty)}=u^m$ makes sense.

\section{Existence of solutions}
\label{sec:bb}

In  this Section we prove existence and some properties of solutions to problem \eqref{06112}.
\begin{theorem}
\label{06116}
Let assumption \eqref{A0} be satisfied. Then there exists a solution $(u,w)$ to problem \eqref{06112}.
\end{theorem}

The proof of Theorem \ref{06116} is divided in several steps. Let
us fix some notations. For any $R>0$ we denote by $
B_R^{N}=B^N(0,R)$ the Euclidean open ball in $\Rr^N$; we  set
\begin{align*}
\Omega_R &:= \Omega\cap B_R^{2}\,,\\
\Gamma_R &:=\partial \Omega_R   \cap \{y=0\} \equiv B_R^1\,,\\
\Sigma_R &:= \partial\Omega_R \cap \{y>0\}.
\end{align*}
If $f\in W^{1,2}(\Omega_R)$, we denote by $f |_{\Gamma_R}, f|_{\Sigma_R}$ the trace of $f$ on $\Gamma_R$ and $\Sigma_R$, respectively.

For any fixed $\epsilon>0$, $R>0$  and $g\in L^\infty(\Gamma_R)$ we consider
the following auxiliary elliptic problem in $\Omega_R$:
\begin{equation}
\label{061111}
    \begin{cases}
    \Delta v_R = 0 & (x,y)\in \Omega_R\\
    \displaystyle-\epsilon\frac{\partial v_R}{\partial y} + \rho\, (v_R)^\frac{1}{m}= \rho\, g & x\in \Gamma_R\\
    v_R=0 & (x,y)\in \Sigma_R.
    \end{cases}
\end{equation}

\begin{definition}
\label{231101}
A \emph{solution} to problem \eqref{061111} is a pair of functions
$(z_R, v_R)$ such that
\begin{itemize}
\item{} $z_R \in L^\infty(\Gamma_R)$;
\item{} $v_R \in W^{1,2}(\Omega_R)\cap L^\infty(\Omega_R)$;
\item{} $v_R\big{|}_{\Gamma_R}=z_R^m, \; v_R\big{|}_{\Sigma_R}=0;$
\item{} for any $\varphi\in C^1(\overline{\Omega_R}), \varphi=0$ on $\Sigma_R$ there holds
    \begin{equation}
    \label{061150}
    \int_{\Omega_R} \langle \nabla v_R, \nabla \varphi \rangle \,\dd x\, \dd y
    + \frac{1}{\epsilon}\int_{\Gamma_R} \rho \left(z_R- g \right) \varphi \,\dd x=0.
     \end{equation}
\end{itemize}
\end{definition}

In the following Lemma we establish existence, uniqueness and some
properties of solutions to problem \eqref{061111}.
\begin{lemma}
\label{061110}
Let $R>0$ and $\epsilon>0$ be fixed. For any $g\in L^{\infty}(\Gamma)$, there exists a unique solution $(z_R,v_R)$ to the problem \eqref{061111}.
Furthermore the following properties hold: 
\begin{itemize}
\item[i.]
\underline{Contractivity:}
the mapping
\[
g\mapsto \frac{1}{\rho}z_R
\]
is a contraction in the norm of $L^{1}(\Gamma_R)$, namely, if $v_R$ and $\tilde{v}_R$ are solutions of \eqref{061111} corresponding to $g$ and $\tilde{g}$ respectively, then
\begin{equation}
\label{14111}
\int_{\Gamma_R} \rho \left( z_R- \tilde{z}_R\right)_{+} \, \dd x
\leq \int_{\Gamma_R} \rho \left[ g - \tilde{g} \right]_{+} \, \dd x;
\end{equation}
in particular, if $g\geq 0$ then $z_R\geq 0$ in $\Omega_{R}$;

\item[ii.] \underline{Uniform boundedness:}
for any $\epsilon>0$ and any $R>0$
\begin{equation}
\label{14113}
0\leq v_R \leq  \|g\|_{L^\infty(\Gamma)}^{m} \quad\text{in }\Omega_R,
\end{equation}
provided $g\ge 0$;

\item[iii.] \underline{Monotonicity:}
if $R'>R$ then $v_{R'}\geq v_{R}$ in $\Omega_{R}$, provided $g \geq 0$.
\end{itemize}
\end{lemma}

\begin{proof} Let $\tilde H^1(\Omega_R):=\{f\in H^1(\Omega_R)\,:\, f|_{\Sigma_R}=0\}$. 
Existence of solutions easily follows by standard arguments, since
the functional $J: \tilde H^1(\Omega_R)\to \Rr$ defined by
\[
J(v):= \frac{1}{2}\int_{\Omega_R} \left|\nabla v\right|^2\, \dd x\, \dd y + \frac{m}{m+1}\int_{\Gamma_R} \rho\, v^\frac{m+1}{m}\, \dd x - \int_{\Gamma_R} \rho\, v\,g\, \dd x
\]
is coercive in $\tilde H^1(\Omega_R)$; in fact, for suitable constants $C_1=C_1(R)>0$ and $C_2=C_2(R)>0$, we have
\begin{multline*}
J(v)        \geq \frac{1}{2}\int_{\Omega_R} \left|\nabla v\right|^2\, \dd x\, \dd y
        + \left(\min_{\Gamma_R}\rho\right) \frac{m}{m+1}\int_{\Gamma_R} v^\frac{m+1}{m}\, \dd x\\
        - \left(\max_{\Gamma_R }\rho\right) \int_{\Gamma_R} v\,g\, \dd x
        \geq C_1 \|v\|^2_{H^1(\Omega_R)} - C_2\| v\|_{H^1(\Omega_R)}.
\end{multline*}

Now, to show contractivity, take a smooth monotone approximation
$p$ of the sign function, and set $\varphi(x):= p[v_R(x) -
\tilde{v}_R(x)],\, x\in \Omega_R$. Clearly, such a $\varphi$ can
be used as a test function in the weak formulation of Definition
\ref{231101}. We get
\begin{multline*}
0= \int_{\Omega_R} p'(v_R - \tilde{v}_R) |\nabla (v_R - \tilde{v}_R)|^2 \,\dd x\, \dd y \\
    + \frac{1}{\epsilon}\int_{\Gamma_R} \rho \left(z_R - \tilde{z}_R \right) p[(z_R)^m- (\tilde{z}_R)^m]\,\dd x
    -  \frac{1}{\epsilon}\int_{\Gamma_R} \rho \left(g - \tilde{g} \right) p[(z_R)^m - (\tilde{z}_R)^m] \,\dd x.
\end{multline*}
Passing to the limit in the approximation $p$, we discover
\[
\int_{\Gamma_R} \rho \left(z_R - \tilde{z}_R\right)_{+}\,\dd x
\leq \int_{\Gamma_R} \rho \left(g - \tilde{g} \right)\D{sgn}(z_R^m - \tilde{z}_R^m) \,\dd x
\]
which immediately gives \eqref{14111}. Furthermore, if $g\geq 0$,
\eqref{14111} implies $z_R\geq 0$ in $\Gamma_R$. From
\eqref{14111} uniqueness can be easily deduced.

To prove uniform boundedness, set $\tilde g:= \|g\|_{L^\infty(\Gamma)}$;  thus the function $\tilde{v}_R:= (\tilde{g})^m$ verifies:
\[
    \begin{cases}
     \Delta\tilde{v}_R = 0 & (x,y)\in \Omega_R\\
    \displaystyle-\epsilon\frac{\partial \tilde{v}_R}{\partial y} + \rho\,(\tilde{v}_R)^\frac{1}{m}= \rho\, \tilde g \geq \rho\, g & x\in \Gamma_R\\
 \tilde{v}_R\geq 0 & (x,y)\in \Sigma_R.
    \end{cases}
\]
Then, by \eqref{14111}, $z_R \leq \tilde{z}_R$ in $\Gamma_R$.
Thus, the function $\zeta_R:= \tilde{v}_R -v_R$  satisfies the
problem
\begin{equation}
\label{291102}
    \begin{cases}
    \Delta \zeta_R = 0 & (x,y)\in \Omega_R\\
    \zeta_R\geq 0 & (x,y)\in \Sigma_R \cup \Gamma_R.
    \end{cases}
\end{equation}
The comparison principle then implies $\zeta_R\geq 0$ in $\Omega_R$, that is
\[
v_R\leq \|g\|_{L^\infty(\Gamma)}^{m}, \quad \text{in $\Omega_R$}.
\]
On the other hand, since $\underline{z}\equiv 0$ is a subsolution to problem \eqref{291102}, $v_R\geq 0$; \eqref{14113} is then proved.

To establish {\em iii.} we consider $v_{R'}$ solution of
\eqref{061111} in $\Omega_{R'}$. Since, by {\em ii.}, $v_{R'}$ is
nonnegative in $\Omega_{R'}$, we have
\[
    \begin{cases}
    \Delta v_{R'} = 0 & (x,y)\in \Omega_R\\
    \displaystyle-\epsilon\frac{\partial v_{R'}}{\partial y} + \rho\, (v_{R'})^\frac{1}{m}= \rho\, g & x\in \Gamma_R\\
    v_{R'}\geq 0 & (x,y)\in \Sigma_R.
    \end{cases}
\]
Then, by \eqref{14111},
\[
 \int_{\Gamma_R} \left( z_R - z_{R'}\right)_{+} \rho\, \dd x
= 0;
\]
this implies $z_{R'}\geq z_{R}$ on $\Gamma_R$. To conclude,  observe that the function $\eta_R:= v_{R'}-v_R$ satisfies,
\[
    \begin{cases}
    \Delta \eta_R = 0 & (x,y)\in \Omega_R\\
    \eta_R\geq 0 & (x,y)\in \Sigma_R \cup \Gamma_R.
    \end{cases}
\]
Hence, by comparison, $ v_{R'}\geq v_{R}$ on $\Omega_R$.
\end{proof}

Next step in the proof of Theorem \ref{06116} consists in studying the following evolutive problem in $\Omega_R$:

\begin{equation}
\label{06118}
    \begin{cases}
    \Delta w_R = 0 & (x,y)\in \Omega_R, \quad t>0\\
    w_R = 0 & (x,y)\in \Sigma_R, \quad t>0\\
    \displaystyle \frac{\partial w_R}{\partial y}= \rho\, \frac{\partial\left[ (w_R)^\frac{1}{m}\right]}{\partial t} & x\in \Gamma_R, \quad t>0\\
     w_R=u_0^m & x\in \Gamma_R,  \quad t=0.
    \end{cases}
\end{equation}

\begin{definition}
\label{231103} A \emph{solution} to problem \eqref{06118} is a
pair of functions $(u_R, w_R)$ such that
\begin{itemize}
\item{} $u_R\in
L^{\infty}(\Gamma_R\times (0,\infty))\cap C([0,\infty); L^1(\Gamma_R))$, $ u_R\ge 0$;
\item{} $w_R\in
L^{\infty}(\Omega_R\times (0,\infty))$,  $w_R\ge 0$, $|\nabla w_R|\in
L^2(\Omega_R\times (0, +\infty))$;
\item{} $w_R\big{|}_{\Gamma_R\times (0,\infty)}=u_R^m$, $w_R\big{|}_{\Sigma_R\times(0,\infty)}=0$;
\item{} for any $T>0$ and any $\psi\in C^{2,1}_{x,t}(\overline{\Omega_R}\times [0,T])$, $\psi\equiv 0$ in $\partial\Sigma_R \times (0,T)$ there holds
    \begin{multline}
    \label{231104}
- \int_0^T \int_{\Omega_R} \langle \nabla w_R, \nabla \psi \rangle \,\dd x\, \dd y\, \dd t
+ \int_0^T \int_{\Gamma_R}\rho \,  u_R \partial_t \psi \, \dd x\, \dd t  \\
= \int_{ \Gamma_R} \rho(x) \Bigl[ u_R(x,T)\psi(x,0,T) - u_0(x)\psi(x,0,0)\Bigr] \,\dd x\,.
     \end{multline}
\end{itemize}
\end{definition}

Observe that, by standard results, for any $g\in L^\infty(\Gamma_R)$ there exists a unique
solution $v\in W^{1,2}(\Omega_R)\cap L^\infty(\Omega_R)$ of problem
\[
\begin{cases}
\Delta v =0 & \text{in }\Omega_R\\
v=0 & \text{on }\Sigma_R\\
v(x,0)= g(x) & \text{on }\Gamma_R;
\end{cases}
\]
here boundary conditions are meant in the sense of trace. We shall denote such a solution
by $\D{E}_R(g)$ since it can be regarded as the harmonic extension of $g$ to $\Omega_R$, completed with
homogeneous zero Dirichlet boundary condition on $\Sigma_R$. Indeed, by standard elliptic regularity results, $v\in C^2(\Omega_R)\cap C(\Omega_R\cup\Sigma_R)$ and
$v(x)=0$ for all $x\in \Sigma_R$.

\begin{proposition}
\label{06117} Let assumption \eqref{A0} be satisfied. Then for any
$R>0$ there exists a solution $(u_R, w_R)$ of problem
\eqref{06118}.
\end{proposition}


\begin{proof}
We introduce a  time-discretized version of \eqref{06118}. Fix any
$T>0$. For any $n\in \Nn$ we divide the time interval $[0,T]$ in
$n$ subintervals of length $\epsilon= T/n$ and endpoints
$[(k-1)\epsilon, \, k\epsilon]$ for $k=1, \dots, n$. For any $k=1,
\dots, n$, by Lemma \ref{061110} with $g=u_R^{\epsilon, k-1}$, there exists a unique solution $(u_R^{\epsilon, k}, w_R^{\epsilon, k})$
to the problem
\begin{equation}
\label{06119}
    \begin{cases}
    \Delta w^{\epsilon, k}_R = 0 & (x,y)\in \Omega_R\\
    w^{\epsilon, k}_R = 0 & (x,y)\in \Sigma_R\\
    \displaystyle \epsilon \, \frac{\partial w^{\epsilon, k}_R}{\partial y}
        = \rho\, \left( u^{\epsilon, k}_R - u^{\epsilon, k-1}_R\right)& x\in \Gamma_R\\
     u^{\epsilon, 0}_R=u_0 & x\in \Gamma_R.
    \end{cases}
\end{equation}
The solution $(u^{\epsilon, k}_R, w^{\epsilon, k}_R)$ satisfies
\[
u^{\epsilon, k-1}_R = \left.\left(\left(w^{\epsilon, k}_R\right)^\frac{1}{m}\right)\right|_{\Gamma_R}, \qquad
w^{\epsilon, k}_R = \D{E}_R\left(\left(u^{\epsilon, k}_R\right)^m\right).
\]


We can rewrite the problem on $\Gamma_R$ in \eqref{06119} as
\[
\begin{cases}
u^{\epsilon, k}_R + \epsilon A (u^{\epsilon, k}_R) = u^{\epsilon, k-1}_R, & x\in \Gamma_R, \quad k=1, \dots, n\\
u^{\epsilon, 0}_R = u_0, & x\in \Gamma_R,
\end{cases}
\]
where $A: \mathcal{D}(A)\subset L^1(\Gamma_R) \to L^{\infty}(\Gamma_R)$ is the operator defined
\[
A(v):= - \frac{1}{\rho}\, \left.\left( \frac{\partial \D{E}_R(v^m)}{\partial y} \right)\right|_{\Gamma_R},
\]
with domain
\begin{multline*}
\mathcal{D}(A):= \Big\{ v\in L^{\infty}(\Gamma_R):
\quad A(v)\in L^{\infty}(\Gamma_R), \quad
\|v\|_{L^{\infty}(\Gamma_R)}\leq \|u_0\|_{L^{\infty}(\Gamma_R)}\Big\}.
\end{multline*}
The operator $A$ satisfies the following properties. For any $\epsilon>0$ the mapping
$(I+\epsilon A)$ is one-to-one from $\mathcal{D}(A)$ onto a subspace $\mathcal{R}_\epsilon(A)\subseteq L^{\infty}(\Gamma_R)$.
In fact, by Lemma \ref{061110},  for any $\epsilon>0$ there exists a unique solution of
\begin{equation}
\label{061113}
    \begin{cases}
    \Delta w^{\epsilon, k}_R = 0 & (x,y)\in \Omega_R\\
    \displaystyle  u^{\epsilon, k}_R - \epsilon  \frac{1}{\rho} \frac{\partial w^{\epsilon, k}_R}{\partial y} = u^{\epsilon, k-1}_R& x\in \Gamma_R\\
    w^{\epsilon, k}_R = 0 & (x,y)\in \Sigma_R;
    \end{cases}
\end{equation}
in addition, by \eqref{14111}, the inverse mapping
\[
(I+\epsilon A)^{-1} : \mathcal{R}_\epsilon(A) \to L^\infty(\Gamma_R)
\]
is a contraction with respect to the norm of $L^{1}(\Gamma_R)$.

A second property satisfied by $A$ is the following rank condition: for any $\epsilon>0$
\[
\mathcal{R}_\epsilon(A)= L^{\infty}(\Gamma_R) \supseteq \overline{ L^{\infty}(\Gamma_R) } = \overline{\mathcal{D}(A)}
\quad\text{for any }\epsilon>0.
\]

The validity of the two properties above permits to invoke the
Crandall--Liggett theorem (see \cite{CL})  and to infer that
$u^{\epsilon, k}_R$ converges in $L^1(\Gamma_R)$ to some function
$u_R\in C([0,T]; L^1(\Gamma_R))$, as $\epsilon \to 0$.
Furthermore, by \eqref{14113} and \eqref{A0}-(ii),
\[
0\leq w^{\epsilon, k}_R \leq \|u_0\|_{L^\infty(\Rr)}^{m}
\text{ in $\Omega_R$, uniformly with respect to $\epsilon$, $k$ and $R$.}
\]
Then $w^{\epsilon, k}_R$ converges weakly-* to some function $w_R\in L^\infty(\Omega_R\times [0,T])$ as $\epsilon \to 0$.

Using $w^{\epsilon, k}_R$ as a test function for \eqref{06119}, integrating by parts, and applying Young's inequality we obtain
\begin{equation}
\label{261101} 0\leq \epsilon \int_{\Omega_R} \left|\nabla w^{\epsilon,
k}_R\right|^2\, \dd x\, \dd y \leq \frac{1}{1+m}\int_{\Gamma_R } \rho
\left[ \left(u^{\epsilon, k-1}_R\right)^{m+1} - \left(u^{\epsilon,
k}_R\right)^{m+1} \right]\, \dd x.
\end{equation}

Then, adding for $k$ from $1$ to $n$ and passing to the limit as $\epsilon\to 0$ we discover
\[
\int_0^T \int_{\Omega_R} |\nabla w_R|^2 \, \dd x\, \dd y\, \dd t
\leq \frac{1}{1+m}  \int_{\Gamma_R} \rho\,\left (u_0\right)^{m+1} \dd x.
\]
Thus, $w_R\in L^2([0,T]; H^1(\Omega_R))$. Moreover, it is easily seen that $w_R\big{|}_{\Sigma_R\times (0,\infty)}=0$.
Concerning the limit function $u_R$, we observe that, by \eqref{261101},
\[
\int_{\Gamma_R } \rho\, \left(u^{\epsilon, k}_R\right)^{m+1}\,  \dd x \leq
\int_{\Gamma_R } \rho\,\left(u^{\epsilon, k-1}_R\right)^{m+1}\,  \dd x \leq
\int_{\Gamma_R } \rho\,u_0^{m+1}\,  \dd x,
\]
for any $\epsilon>0$; then for $\epsilon\to 0$, we get
\[
\int_{\Gamma_R }\rho\,(u_R(x, t))^{m+1}\,  \dd x \leq
\int_{\Gamma_R } \rho\, u_0^{m+1}\,  \dd x, \quad\text{for any
}t\in[0,T].
\]

Now, by choosing appropriate test functions as in \cite{RoV} and
taking into account \eqref{061150} we have for any $k=1, \dots,
n$,
    \begin{equation}
    \label{261104}
    \int_{\Omega_R} \langle \nabla w^{\epsilon, k}_R, \nabla \psi \rangle \,\dd x\, \dd y
    = \frac{1}{\epsilon}\int_{\Gamma_R} \rho \left[u^{\epsilon, k-1}_R - u^{\epsilon, k}_R \right] \psi\,\dd x.
     \end{equation}
Integrating on $[t_{k-1}, t_k]$ and then adding over $k$, we can rewrite the right hand side of equality above as
\begin{align}
& \frac{1}{\epsilon} \sum_{k=1}^n \int_{t_{k-1}}^{t_k} \int_{\Gamma_R} \rho \left[u^{\epsilon, k}_R(x,t) - u^{\epsilon, k}_R (x,t) \right] \psi(x,0,t) \,\dd x \,\dd t \nonumber\\
= & \int_0^T \int_{\Gamma_R} \rho \, u^{\epsilon}_R(x,t) \frac{\psi(x, 0, t+\epsilon) - \psi(x, 0, t)}{\epsilon} \, \dd x\, \dd t \nonumber\\
& + \frac{1}{\epsilon} \int_0^\epsilon \int_{\Gamma_R}\rho\,  u_0(x) \psi(x, 0, t)\, \dd x\, \dd t - \frac{1}{\epsilon} \int_{T-\epsilon}^ T \int_{\Gamma_R}\rho\,  u^\epsilon_R (x,T) \psi(x, 0, t)\, \dd x \, \dd t , \label{261105}
\end{align}
where $u^\epsilon_R := \{u^{\epsilon, 1}_R, \dots, u^{\epsilon, n}_R\}$.
By sending $\epsilon$ to $0$, \eqref{261104} and \eqref{261105}  give \eqref{231104}.
\end{proof}

In the next Proposition we list some properties satisfied by solutions of \eqref{06118}.

\begin{proposition}
\label{14115}
Let $(u_R,w_R)$ be a solution of \eqref{06118}. Then the following properties are satisfied:
\begin{itemize}
\item[i.] If $u_R$ and $\tilde{u}_R$ are solutions of \eqref{06118} corresponding to $u_0$ and $\tilde{u}_0$ respectively, then
\begin{equation}
\label{14114}
\int_{\Gamma_R} \left[ u_R (x,t)- \tilde{u}_R(x,t))\right]_{+} \rho\, \dd x
\leq \int_{\Gamma_R} \left[ u_0 - \tilde{u}_0 \right]_{+} \rho\, \dd x;
\end{equation}
in particular, if $u_0\geq 0$ then $u_R(x,t)\geq 0$ for every
$x\in \Gamma_R$ and every $t>0$;
\item[ii.] There exists a constant $C$ independent of $R$ such that $0\leq w_R\leq C$ for every $x\in \Omega_R$ and every $t>0$;
\item[iii.] For any $R'>R$, $w_{R'}(x,t) \geq w_{R}(x,t)$ for every $x\in \Omega_{R}$ and every $t>0$;
\item[iv.] The inequality
\[
\rho(x)\left[(m-1)t\partial_t w_R + mw_R\right]\geq 0
\]
holds in the sense of distributions in $\Omega_R \times (0, \infty)$;
\item[v. ]$\partial_t u \in L^1_{loc} ((0, \infty); L^1(\Gamma))$;
\item[vi. ] For any $0\leq \tau <T$,
\begin{multline}
\label{261107}
 \int_\tau^T \int_{\Omega_R} |\nabla w_R|^2\, \dd x\, \dd y \,\dd t
+ \frac{1}{m+1}\int_{\Gamma_R } \rho \, (u_R)^{m+1}(x, T) \,\dd x \\
= \frac{1}{m+1}\int_{\Gamma_R } \rho \, (u_R)^{m+1}(x, \tau) \,\dd x;
\end{multline}
\item[vii.] For any $0\leq \tau <T$ and any cutoff $\eta\in C^2_0 (\bar\Omega)$ with $0\le \eta\le 1$ and
$\D{supp}(\eta)\subset B_R $,
\begin{multline}
\label{291104}
 \int_\tau^T \int_{\Omega_R} |\nabla w_R|^2 \, \eta^2\, \dd x\, \dd y \,\dd t
+ \frac{2}{m+1}\int_{\Gamma_R } \rho(x) \, (u_R)^{m+1}(x, T) \,\eta^2\,\dd x \\
= \frac{2}{m+1}\int_{\Gamma_R } \rho(x) \, (u_R)^{m+1}(x,
\tau)\,\eta^2 \,\dd x - 2  C^2\int_0^T \int_{\Omega_R}
|\nabla\eta|^2 \, \dd x\, \dd y \,\dd t.
\end{multline}

\end{itemize}
\end{proposition}

\begin{proof}
Properties {\em i.}--{\em iii.} are inherited by $(u_R, w_R)$ from those listed in Lemma \ref{061110} for solutions to the auxiliary elliptic problem. In particular  {\em i.} is a consequence of \eqref{14111}. Property {\em ii.} follows taking into account that $w_R$ is null on $\Sigma_R$ and previous property {\em i.}; property {\em iii.} follows by Lemma \ref{061110}-{\em iii.}

Properties {\em iv.}--{\em vi.} can be obtained arguing as in
\cite[Sections 3 and 5]{V1}. Moreover, it is direct to check that
\begin{gather*}
 \int_\tau^T \int_{\Omega_R} |\nabla w_R|^2 \, \eta^2\, \dd x\, \dd y \,\dd t
+ \frac{1}{m+1}\int_{\Gamma_R } \rho(x) \,(u_R)^{m+1}(x, T) \,\eta^2\,\dd x \nonumber \\
= \frac{1}{m+1}\int_{\Gamma_R } \rho(x) \, (u_R)^{m+1}(x, \tau)\,\eta^2 \,\dd x
- 2  \int_\tau^T \int_{\Omega_R} \eta \, w_R \langle \nabla\eta , \nabla w_R\rangle \, \dd x\, \dd y \,\dd t.
\label{291106}
\end{gather*}
Using Young's inequality, we estimate the last term in the computation above as follows:

\begin{gather*}
\left| 2  \int_0^T \int_{\Omega_R} \eta \, w_R \langle \nabla\eta , \nabla w_R\rangle \, \dd x\, \dd y \,\dd t\right|
\leq 2  \int_0^T \int_{\Omega_R}  \bigl(\left|\nabla w_R\right| \eta\bigr) \bigl(w_R  \|\nabla\eta|
\bigr)\, \dd x\, \dd y \,\dd t\\
\leq \frac{1}{2} \int_0^T \int_{\Omega_R}  |\nabla w_R|^2 \eta^2 \, \dd x\, \dd y \,\dd t
+ 2 \int_0^T \int_{\Omega_R}  |w_R|^2  |\nabla\eta|^2 \, \dd x\, \dd y \,\dd t.
\end{gather*}

The last inequality, together with the fact that, by {\em ii.}, $w_R^2\leq C^2$, 
 gives \eqref{291104}.
\end{proof}

\begin{proof}[Proof of Theorem \ref{06116}]
By Proposition \ref{06117}, for any $R>0$ there exists a solution
$(u_R, w_R)$ of \eqref{06118}. Moreover, thanks to the uniform
boundedness and monotonicity of $w_R$ (properties {\em ii.} and
{\em iii.} of Proposition \ref{14115}), there exist the limits
 \begin{gather*}
 \lim_{R\to \infty} u_R=: u \quad\text{in $\Gamma\times (0,\infty)$},\\
 \lim_{R\to \infty}w_R =: w \quad\text{in $\Omega\times (0,\infty)$};
 \end{gather*}
again by Proposition \ref{14115}-{\em ii.}, $w\in
L^{\infty}(\Omega\times (0,\infty))$ and $u\in
L^{\infty}(\Gamma\times (0,\infty))$.

We shall prove that $(u,w)$ is a solution to problem \eqref{06111}. To see this, take $T>0$, $D$ and $\psi$ as in Definition \ref{06114}; then select $R_0>0$ large enough, such that $D\subset \Omega_{R_0}$. By Definition \ref{231103}, for every $R>R_0$ we have:
    \begin{multline}
    \label{03121}
- \int_0^T \int_{D} \langle \nabla w_R, \nabla \psi \rangle \,\dd x\, \dd y\, \dd t
+ \int_0^T \int_{\partial D\cap \Gamma_R}\rho \,  u_R \partial_t \psi \, \dd x \dd t  \\
= \int_{\partial D\cap  \Gamma_R} \rho(x) \Bigl[ u(x,T)\psi(x,0,T) - u_0(x)\psi(x,0,0)\Bigr] \,\dd x.
     \end{multline}

Furthermore, by Proposition \ref{14115}-{\em vii.}, since $u\in L^\infty(\Gamma\times (0,\infty))$ and $u_0\in L^\infty(\Gamma)$,
for each compact set $K\subset \bar \Omega$, there exists a positive constant $C$, depending on $K$ but independent of $R$ and $T>0$, such that
\[
\int_0^T\int_K |\nabla w_R|^2 \leq C \qquad\text{in $K$}.
\]
By usual compactness arguments, $\nabla w\in L^2_{loc}\big(\overline\Omega\times(0,\infty)\big)$, so $w|_{\Gamma}$ is well-defined; moreover,
letting $R\to \infty$ in \eqref{03121} we get
    \begin{multline}
    \label{03122}
- \int_0^T \int_{D} \langle \nabla w, \nabla \psi \rangle \,\dd x\, \dd y\, \dd t
+ \int_0^T \int_{\partial D}\rho \,  u\partial_t \psi \, \dd x \dd t  \\
= \int_{\partial D} \rho(x) \Bigl[ u(x,T)\psi(x,0,T) - u_0(x)\psi(x,0,0)\Bigr] \,\dd x.
     \end{multline}
Hence, it is direct to check
that $(u,w)$ is also a solution in the sense of Definition
\ref{06114}.
\end{proof}

\begin{remark}
\label{03128}
By construction, the solution $(u,w)$ given in the proof of Theorem \ref{06116} is minimal, namely, if $(\tilde{u}, \tilde{w})$ is a solution to \eqref{06112} with $\tilde{u}\geq 0$ and $\tilde{w}\geq 0$,  then $\tilde{u}\geq u$ and $\tilde{w}\geq w$.
\end{remark}

\begin{remark}\label{Rem2a}
$(i)$ As already pointed out in the Introduction, our construction of
the solution is slightly different from those in \cite{V1} and
\cite{V2}. In fact, the strategy in \cite{V1} consists in passing
first to the limit as $R\to \infty$ in the discretization
\eqref{06119} to obtain a solution of an elliptic problem in the
whole of $\Omega$, and then in sending the discretization
parameter $\epsilon$ to zero. Our approach consists instead in
passing first to the limit as $\epsilon\to 0$ in the discretized
elliptic problem, in order to get, for any $R>0$, a solution of
the parabolic problem \eqref{06118} in $\Omega_R$. Finally sending
$R\to \infty$ we get a solution in the whole space. 
A similar strategy has been implemented in \cite{V2}, where unbounded cylinders $B^N_R\times \Rr^+$, instead of bounded domains $\Omega_R$ considered here, have been used to invade $\Omega$.  Note that the use of $\Omega_R$
will be expedient in proving Theorem \ref{tuni} in the following
Section. Indeed, we shall construct a suitable family
of solutions to elliptic problems in bounded domains $\Omega_R$.

\smallskip

$(ii)$ Introducing a varying density $\rho$ does not bring any additional technical difficulty
to the proof of existence of solutions to problem \eqref{06112} compared with those 
in \cite{V1}-\cite{V2}, in which $\rho$ is assumed identically $1$. 

\smallskip

$(iii)$ The restriction $N=1$ does not play any role yet and all results of Section \ref{sec:bb} can be easily adapted to show existence of very weak solutions in any space dimension.

\smallskip

$(iv)$ Concerning the regularity required for the initial data,
we should remark that in \cite{V1} and \cite{V2} $u_0$ is assumed to
belong to $L^1$, whereas in \cite{RV1} and \cite{RV2} - where the
(non fractional) porous medium equation with varying density has
been studied - it is supposed to belong to $L^1_\rho(\Gamma)$. 
As a consequence, the solution considered in
\cite{V1} and \cite{V2} satisfies $|\nabla w|\in L^2(\Omega \times
(\tau, \infty))$ for any $\tau>0$. We assume instead $u_0$
nonnegative, continuous and bounded, so $u_0$ is only in
$L^1_{loc}(\Gamma)\equiv L^1_{\rho, loc}(\Gamma)$, but in general
$u_0\not \in L^1_{\rho}(\Gamma)$. Hence we can infer that $\nabla
w \in L^2_{loc}\big(\overline\Omega\times(0,\infty)\big )$, as
required in our Definition \ref{06114}, but in general $\nabla w
\not\in L^2\big(\Omega\times(\tau,\infty)\big )$.

\smallskip

$(v)$ 
%
We emphasize that, unlike what have been discussed in $(ii)$ and $(iv)$ above, the presence of  a variable density and the assumptions on the initial data shall entail significant differences in the proof of uniqueness of solutions, relative to those in the case in which $\rho\equiv 1$ and $u_0\in L^1$; see also Remark \ref{Remuni}. Moreover, the hypothesis $N=1$ will be essential in our proof of uniqueness; see Remark \ref{Rem1a}. 
\end{remark}

\section{Uniqueness of solutions}
\label{sec:cc}
The scope of this Section is to establish the following

\begin{theorem}\label{tuni}
Let assumption \eqref{A0} be satisfied. Then problem \eqref{06111}
admits at most one solution $(u, w)$ with $u\ge 0$, $w \ge 0$.
\end{theorem}

To prove Theorem \ref{tuni} we need some preliminary materials. To
begin with, observe that the function
\begin{equation}
\label{03127} \Theta(x,y):= - \frac 1{\pi} \log |(x,y)|,
\quad(x,y)\in \Omega
\end{equation}
satisfies
\begin{equation}
\label{03124}
\begin{cases}
\Delta \Theta = 0 &\text{in }\Omega\\
-\partial_y \Theta = \delta_0 &\text{in }\Gamma,
\end{cases}
\end{equation}
where $\delta_0$ is the delta distribution concentrated at the
origin. As a consequence, (see, $e.g.$, \cite{CS1} and \cite{CS2})
for any $F\in C^{\infty}_0(\Gamma),\, F\ge 0$, the function 
\[
w(\cdot, y):=\Theta(\cdot, y)* F
\]
is, up to additive constants, the unique bounded solution of
\begin{equation}
\label{03123}
\begin{cases}
\Delta U = 0 &\text{in }\Omega\\
\displaystyle-\frac{\partial U}{\partial y} = F &\text{in }\Gamma;
\end{cases}
\end{equation}
for a detailed proof of the fact that $w$ solves problem
\eqref{03123} we refer, $e.g.$, to \cite{CP}.
\medskip
\smallskip

The following Lemma will play a crucial role in the proof of
Theorem \ref{tuni}.

\medskip

\begin{lemma}
\label{l1u}
Let assumption \eqref{A0} be satisfied, and $R_0>0$ be fixed. Let $\psi_R$ be the
solution of the problem:
\begin{equation}\label{e1}
    \begin{cases}
    \Delta U = 0 & (x,y)\in \Omega_R\\
    \displaystyle - \frac{\partial U}{\partial y} = F & x\in \Gamma_R\\
    U=0 & (x,y)\in \Sigma_R
    \end{cases}
\end{equation}
where $R>R_0$, $F\in C^\infty(\Gamma)$,
$F\ge 0$, $\D{supp}\,F\subseteq \Gamma_{R_0}$.
Then the following statements hold true:

\begin{itemize}
\item[(i)] For any $R_1,R_2 \in (R_0, \infty)$, $R_1\le R_2$ we have
\begin{equation}
\label{e2u}
0< \psi_{R_1}\le\psi_{R_2}\quad \textrm{in } \Omega_{R_1}; 
\end{equation}

\item[(ii)] There exists
a constant $M>0$ depending on $R_0$ such that for any $R>2 R_0$ we
have
\begin{equation}
\label{e3u} -  \frac M{R\left(\log R-\log R_0 \right)}\le
\frac{\partial\psi_R}{\partial \vec\nu_R}< 0 \quad\textrm{on
}\Sigma_R,
\end{equation}
$\vec\nu_{R}$ denoting the outer normal to $\Sigma_R$.

\end{itemize}

\end{lemma}
\medskip

\begin{proof}
$(i)$ By the maximum
principle $\psi_R>0$ in $\Omega_R$ for any $R>R_0$; hence the
function $\psi_{R_1}\,-\, \psi_{R_2}\,\;(R_0<R_1<R_2)$ is a
subsolution of problem
\[
    \begin{cases}
    \Delta U = 0 & (x,y)\in

    \Omega_{R_1}
    \\
    \displaystyle - \frac{\partial U}{\partial y} = 0 & x\in \Gamma_{R_1}
    \\
    U=0 & (x,y)\in
    \Sigma_{R_1}.
    \end{cases}
\]
Then again by the maximum principle we get \eqref{e2u}.

\medskip

$(ii)$ Clearly, by the strong maximum principle it follows that,
for any $R>R_0$,
\begin{equation}\label{e9u}
\frac{\partial \psi_R}{\partial \vec \nu_R}< 0\quad \text{on }\Sigma_R.
\end{equation}

Put
\[
\Theta(\cdot, y)*F =: \psi_\infty(x, y), \quad (x, y)\in \bar \Omega,
\]
where $\Theta= \Theta(|(x,y)|)$ is defined in \eqref{03127}.
It is easily seen that for any $R>R_0$ $\psi_\infty$ is a supersolution \eqref{e1}. By comparison principle,
\[ 
\psi_R \le \psi_\infty\quad \text{in }\Omega_R. 
\]
Indeed, it could be shown that
\[
\psi_\infty(x, y) = \lim_{R\to \infty} \psi_R(x,y), \quad (x, y)\in \bar \Omega.
\]

Now, for any $R>R_0$, consider the problem
\begin{equation}\label{e7u}
\begin{cases}
    \Delta U = 0 & (x,y)\in \Omega_R\backslash \overline\Omega_{R_0}\\
    \displaystyle - \frac{\partial U}{\partial y} = 0 & x\in \Gamma_R\backslash\Gamma_{R_0}\\
    U=0 & (x,y)\in\Sigma_{R}\\
    U = M & (x,y)\in \Sigma_{R_0},
    \end{cases}
\end{equation}
where ${\displaystyle M:=\max_{\Sigma_{R_0}}\psi_{\infty}}$.
Define
\[
Z(x,y)\equiv Z(|(x,y)|):=  M \frac{\Theta(|(x,y)|)- \Theta(R)}{\Theta(R_0)-\Theta(R)},\qquad (x,y)\in
\Omega_R\backslash \Omega_{R_0}.
\]
Using \eqref{03124} immediately follows that $Z$ is a
supersolution to problem \eqref{e7u} for any $R> 2R_0$. On the
other hand, since in particular $R>R_0$, $\psi_R$ is a subsolution
to the same problem. By comparison we get for any $R>2 R_0$
\[
Z\ge\psi_R\quad \textrm{in } \Omega_{R}\backslash\Omega_{R_0}.
\]

Consequently, since for any $R> 2R_0$
\[
Z=\psi_R=0\quad \textrm{on}\;\;\Sigma_R,
\]
we can infer that, for any $(x,y)\in
\Sigma_{R}$,
\[
\frac{\partial\psi_R}{\partial\vec\nu_R} \ge \frac{\partial
Z}{\partial\nu_R}=\frac{\partial Z(R)}{\partial r}= -\frac
{M}{R\left(-\log R_0 + \log R \right)};
\]
here $r\equiv r(x,y)=|(x,y)|$. The above inequality, combined
with \eqref{e9u} gives \eqref{e3u}. This completes the proof.
\end{proof}

Now we can prove Theorem \ref{tuni}.

\begin{proof}
[Proof of Theorem \ref{tuni}.]
Let $(\underline u, \underline w)$ be the solution of problem \eqref{06111} constructed in Theorem \ref{06116};  such solution is minimal as observed in Remark \ref{03128}. Let $(u,w)$ be
any other solution of the same problem and set
\[
K:=\|w\|_{L^{\infty}(\Omega \times (0, \infty))}\, .
\]
We claim that
\begin{equation}\label{e18u}
\int_{0}^T\int_{\Gamma}\big[u^m - \underline{u}^m)\big]\, F\,\dd
x \, \dd t\, =\, 0
\end{equation}
for any $T>0, \,F\in C^{\infty}_0(\Gamma)$.

By the Claim the conclusion follows. In fact, in view of
the arbitrariness of $F$, \eqref{e18u}
implies

\begin{equation}\label{e19u}
u\,=\,\underline{u}\,\quad \textrm{in}\,\, \Gamma\times (0,T),
\end{equation}
whence the conclusion.

It remains to prove the Claim. To this aim, without loss of
generality, we suppose $\D{supp}(F)\subseteq \Gamma_{R_0}$  for some $R_0>0$, $F\ge
0$, $F\not\equiv 0$.
Take $R>2R_0$ and let $\psi_R(x,y)$ be the solution of problem \eqref{e1}. Using $\psi_R$ as a test function for the equation \eqref{06112} (see Definition \ref{06114}) we obtain, for any $\tau>0$,

\begin{align}
0 =
  &- \int_0^\tau \int_{\partial \Omega_R} w \frac{\partial \psi_R}{\partial \vec\nu_R}\dd S\, \dd t
  - \int_{\Gamma_R} \rho(x) \Bigl[ u(x, \tau)- u_0(x)\Bigr] \psi_R(x,0) \,\dd x \nonumber\\
 = & \int_0^\tau \int_{\Gamma_R} w \frac{\partial \psi}{\partial y}\dd x\, \dd t
 - \int_0^\tau \int_{\Sigma_R} w \frac{\partial \psi_R}{\partial \vec\nu_R}\dd x\, \dd y\, \dd t \nonumber\\
 & \qquad - \int_{\Gamma_R} \rho(x) \Bigl[ u(x,\tau)- u_0(x)\Bigr] \psi_R(x,0) \,\dd x \nonumber\\
 = & - \int_0^\tau \int_{\Gamma_R} u^m \, F\, \dd x\, \dd t
  - \int_0^\tau \int_{\Sigma_R} w \frac{\partial \psi}{\partial \vec\nu_R}\dd x\, \dd y\, \dd t \nonumber\\
 & \qquad - \int_{\Gamma_R} \rho(x) \Bigl[ u(x,\tau)- u_0(x)\Bigr] \psi_R(x,0) \,\dd x.\label{27111}
\end{align}
Similarly, using $\psi_R$ as a test function for $\underline{u}$, we get

\begin{multline}
\label{27112}
0 = - \int_0^\tau \int_{\Gamma_R} \underline{u}^m \, F\, \dd x\, \dd t
  - \int_0^\tau \int_{\Sigma_R} \underline{w} \frac{\partial \psi}{\partial \vec\nu_R}\dd x\, \dd y\, \dd t\\
 - \int_{\Gamma_R} \rho(x) \Bigl[ \underline{u}(x,\tau)- u_0(x)\Bigr] \psi_R(x,0) \,\dd x.
\end{multline}

By subtracting \eqref{27112} to \eqref{27111} we have
\begin{multline}
\label{e20u}
\int_{\Gamma_R}\rho(x)\big[u(x,\tau)-\underline{u}(x,\tau)\big]\psi_R(x,0)\,\dd x
+ \int_{0}^{\tau}\int_{\Gamma_R}\big[u^m-\underline{u}^m\big] F(x)\,\dd x\,\dd t\\
=
-\int_{0}^{\tau}\int_{\Sigma_R} \big\{w- \underline{w}\big\}\frac{\partial\psi_R}{\partial\vec \nu_R}\dd x\,\dd t.
\end{multline}
Since $F\ge 0$, $ \psi_R\ge 0 $ and, by minimality of $\underline{u}$, $ u\ge \underline{u}$,
equality \eqref{e20u} with $\tau=T$ gives

\begin{multline}
\label{e21u}
\int_{0}^T\int_{\Gamma_R}\big[u^m -\underline{u}^m\big] F(x)\,\dd x\,\dd t \\
\le \liminf_{R\to \infty}
\left|\int_{0}^T\int_{\Sigma_R} \big\{w-\underline{w}\big\}\frac{\partial \psi_R}{\partial\vec\nu_R}\,\dd x\,\dd t \right|.
\end{multline}
By recalling that $|\Sigma_R|= \pi R/2$ and by using \eqref{e3u} we obtain
\begin{align*}
\left|\int_{0}^T\int_{\Sigma_R}\big\{w-\underline{w}\big\}\frac{\partial\psi_R}{\partial\vec\nu_R}\,\dd x\,\dd t \right|
&\le K\pi R \sup_{\Sigma_R}\left|\frac{\partial \psi_R}{\partial\vec \nu_R}\right|\\
&\le K \pi \frac{M}{\left(-\log R_0 +\log R \right)} \to 0,
\quad\textrm{as }R\to \infty.
\end{align*}
Hence, the right hand side of \eqref{e21u} is zero
and the claim \eqref{e18u} follows.  This concludes the proof.
\end{proof}

\begin{remark}\label{Remuni}
Uniqueness results in \cite{V1} and \cite{V2} cannot be applied to problem \eqref{06111}, as they require $\rho\equiv 1$. Furthermore, even in the case in which $\rho$ is identically $1$, 
the results in \cite{V1} and \cite{V2} cannot be adapted to prove uniqueness of solutions in the sense of Definition \ref{06114}. 
\end{remark}

\begin{remark}\label{Rem1a}
If $N\ge 2$, a result similar to Lemma \ref{l1u}  could be proved. In this case, the estimate \eqref{e3u} must be replaced by the following:
\begin{equation}
\label{e3u_bis}
-  \frac M{R^{N-1}}\le \frac{\partial\psi_R}{\partial \vec\nu_R}< 0,
\quad\textrm{on }\Sigma_R,
\end{equation}
for some $M>0$, for any $R>2R_0$. Hence we cannot get the conclusion in the previous proof.
A similar issue arises for problem \eqref{ea2} when $N\geq 3$; however, methods used in \cite{P1} to get the conclusion in that case cannot be adapted to the present nonlocal situation.
\end{remark}

\begin{remark}
\label{Rem3a}
Note that the proof of uniqueness for problem \eqref{ea2} when
$N=1, 2$ is quite different from the previous one (see \cite{GHP},
\cite{P4}). However, the estimate \eqref{e3u} is crucial in
\cite{GHP} for $N=2$, too. 
Observe finally that the hypothesis $\rho\in L^\infty(\Rr^2)$ used in \cite{GHP} is not required here. Thus our arguments can be adapted in order to prove uniqueness for problem  \eqref{ea2} in the case $N=2$ and $\rho$ possibly unbounded.
\end{remark}

%

\bibliographystyle{plain}
\addcontentsline{toc}{section}{References}
\bibliography{porosoro}

\end{document}